\documentclass[a4paper,12pt,reqno]{amsart}
\usepackage{amsfonts}
\usepackage[colorlinks,linkcolor=blue,citecolor=blue]{hyperref}
\usepackage{latexsym, amssymb, amsmath, amsthm, bbm}
\usepackage[all]{xy}
\usepackage{enumerate}
\usepackage{mathrsfs}
\usepackage{geometry}
\usepackage{pgfplots}

\def \k{\Bbbk}
\def \id{\operatorname{id}}

\usepackage{mathtools}

\numberwithin{equation}{section}
\theoremstyle{plain}
 \newtheorem{theorem}{Theorem}[section]

\newtheorem{lemma}[theorem]{Lemma}
\newtheorem{proposition}[theorem]{Proposition}
\newtheorem{corollary}[theorem]{Corollary}
\newtheorem{definition}[theorem]{Definition}
\newtheorem{example}[theorem]{Example}

\begin{document}
\title[Simple-shifted monoidal derived equivalences]{Simple-shifted monoidal derived equivalences for tensor categories}
\thanks{}
\author[J. Yu]{Jing Yu}
\address{School of Mathematical Sciences, University of Science and Technology of China, Hefei 230026, China}
\email{yujing46@ustc.edu.cn}

\begin{abstract}
In this paper, we investigate monoidal derived equivalences between tensor categories with infinitely many isomorphism classes of simple objects. We introduce the notion of a simple-shifted monoidal derived equivalence and prove that a simple-shifted monoidal derived equivalence gives rise to a tensor equivalence. In addition, we investigate which tensor categories have the property that every monoidal derived equivalence between them is necessarily simple-shifted, and we establish some sufficient conditions for this property. As a corollary, we prove that if a tensor category is monoidally derived equivalent to the finite-dimensional representation category of either the quantized enveloping algebra $U_q(\mathfrak{g})$, the quantized Borel subalgebra $U_q^{\geq 0}(\mathfrak{g})$, or the quantized coordinate ring $\mathcal{O}_{q}(G)$ (for a suitable $q$), it is tensor equivalent to the corresponding category.
\end{abstract}
\date{}
\maketitle

\section{Introduction}
Monoidal triangulated categories arise naturally in stable homotopy theory, algebraic geometry, and representation theory.
In the setting of symmetric monoidal triangulated categories, Balmer developed tensor triangular geometry, providing a systematic framework for studying the prime tensor ideals and supports \cite{Bal05}. Nakano, Vashaw, and Yakimov extended this framework to monoidal triangulated categories without assuming a braiding, establishing a noncommutative theory of spectra and support with applications to finite-dimensional Hopf algebras and small quantum Borel algebras \cite{NVY22}. Subsequent developments address the tensor product property for cohomological support maps \cite{NVY22b} and the relationship between the cohomological support theory and the universal support theory based on the noncommutative Balmer spectra of the stable (triangulated) categories for finite tensor categories \cite{NVY24}. These developments illustrate how the interaction between tensor products and triangulated structures reveals information that is not captured by the triangulated structure alone.

It is natural to ask to what extent an abelian category is determined by its derived category. In general, derived equivalence does not imply equivalence of the underlying abelian categories. Rickard's derived Morita theory \cite{Ric89} provides a systematic framework for understanding such equivalences through tilting complexes. For monoidal categories, however, a derived equivalence preserving the monoidal structure imposes extra constraints and can make reconstruction possible. In algebraic geometry, the term reconstruction theorem refers to results that characterize this phenomenon; see \cite{Bal02,BO01}. For algebras, Aihara and Mizuno \cite{AM17} proved that, among preprojective algebras of Dynkin type, derived equivalence forces Morita equivalence. Zhang and Zhou \cite{ZZ22} showed that an analogous phenomenon holds for finite-dimensional hereditary weak bialgebras. Furthermore, Xu and Zheng \cite{XZ26} showed that, for tensor categories with finitely many isomorphism classes of simple objects, tensor equivalence is equivalent to monoidal triangulated equivalence of their derived categories.
In our previous work \cite{Yu26}, we studied reconstruction theorems for suitable abelian subcategories of tensor categories with infinitely many isomorphism classes of simple objects, using triangulated module equivalences between triangulated module categories.
More recently, Tan \cite{Tan26} established reconstruction for semisimple tensor categories with possibly infinitely many isomorphism classes of simple objects.
However, these results do not settle the general reconstruction problem for nonsemisimple tensor categories with infinitely many isomorphism classes of simple objects.

This paper is devoted to monoidal derived equivalences between tensor categories with infinitely many isomorphism classes of simple objects. Note that in the case of finitely many isomorphism classes of simple objects, Xu and Zheng showed that the presence of a monoidal structure ensures that the heart of the standard
$t$-structure is carried to the heart of the standard
$t$-structure. The central idea of their proof is that a derived equivalence sends the standard $t$-structure to a new $t$-structure, and that any two bounded $t$-structures on the derived category of an abelian category of finite length with finitely many isomorphism classes of simple objects are related \cite[Example 4.5]{CLZ23}. The monoidal structure then exploits this relation to force the new $t$-structure to coincide with the standard one. Our previous work \cite{Yu26} was guided by the same idea.

A natural idea then arises: we still consider the $t$-exact functors in the sense of  \cite[Subsection 2.2]{CHZ19}, but instead of requiring that the heart of the standard $t$-structure be mapped to the heart of the standard $t$-structure, we allow it to be mapped to the heart of another $t$-structure. The question is then whether a monoidal derived equivalence still guarantees a monoidal equivalence. In fact, Tan's study of monoidal derived equivalences for semisimple tensor categories (\cite{Tan26}) is precisely of this type. We also expect the rigidity assumption on the tensor categories to play a genuine role in monoidal derived equivalences, so that such an equivalence forces a tensor equivalence.

In this paper, we first introduce the notion of a simple-shifted functor (see Definition \ref{def:simple-shifted}). It is specified by a triple $(F,\psi,r)$, where
$F:\mathbf{D}^{b}(\mathscr{A}) \to \mathbf{D}^{b}(\mathscr{B})$ is a triangulated functor between bounded derived categories of abelian categories,
$\psi:\mathcal{S}_{\mathscr{A}}\to \mathcal{S}_{\mathscr{B}}$ is a map of sets,
and $r:\mathcal{S}_{\mathscr{A}}\to \mathbb{Z}$ is a map of sets,
subject to the condition $F(S)\cong \psi(S)[r(S)]$ for all $S\in \mathcal{S}_{\mathscr{A}}$. Here $\mathcal{S}_{\mathscr{A}}$ and $\mathcal{S}_{\mathscr{B}}$ denote the sets of isomorphism classes of simple objects in $\mathscr{A}$ and $\mathscr{B}$, respectively. We show that a simple-shifted equivalence ensures that the underlying abelian categories are equivalent (see Proposition \ref{prop:equ}), and moreover that a simple-shifted monoidal equivalence ensures that the underlying multiring categories are tensor equivalent (see Proposition \ref{prop:tensorequ}). We then investigate which tensor categories have the property that every monoidal derived equivalence is necessarily simple-shifted. Along the way, we found that rigidity imposes additional constraints on monoidal derived equivalences; see, for example, Lemma \ref{lem:invertible}.
Our main result is Theorem \ref{thm:Cgamma}, which states that:
\begin{theorem}
Let $\mathscr{C}$ and $\mathscr{D}$ be tensor categories, and let $F: \mathbf{D}^b(\mathscr{C})\rightarrow \mathbf{D}^b(\mathscr{D})$ be a monoidal derived equivalence. Write $\mathscr{C}\cong \bigoplus_{\gamma\in \Gamma} \mathscr{C}^{\gamma},$ where each $\mathscr{C}^{\gamma}$ is a block of $\mathscr{C}$. Suppose that, for each $\gamma \in \Gamma$, at least one of the following conditions holds.
\begin{itemize}
    \item [(i)] For every simple object $S\in \mathscr{C}^\gamma$, there exists a tensor subcategory $\mathscr{E}_{S}$ of $\mathscr{C}$ with finitely many simple objects such that either $S\otimes S^* \in \mathscr{E}_{S}$ or $S^*\otimes S\in \mathscr{E}_{S}$.
    \item[(ii)] For every simple object $S\in \mathscr{C}^\gamma$, we have $S\cong S^*.$
    \item [(iii)] There exists a simple object $S^\gamma\in \mathscr{C}^{\gamma}$ such that every simple object $S\in \mathscr{C}^{\gamma}$ is of the form $S\cong L_S\otimes S^\gamma \otimes R_S$, where $L_S, R_S\in \mathscr{C}$ are invertible objects satisfying $$\min\{i : H^i(F(L_S\otimes R_S)) \neq 0\}=\max\{i : H^i(F(L_S\otimes R_S)) \neq 0\}=0.$$
\end{itemize}
Then $F$ is simple-shifted. Moreover, $\mathscr{C}$ is tensor equivalent to $\mathscr{D}.$
\end{theorem}
It is worth noting that the above theorem covers both the finitely many isomorphism classes of simple objects case of Xu and Zheng \cite{XZ26} and the case of semisimple tensor categories with infinitely many simple objects due to Tan \cite{Tan26}. As a corollary, we further show that a monoidal derived equivalence between pointed tensor categories implies a tensor equivalence between them (see Proposition \ref{prop:pointed}). In addition, the results of Section \ref{section4} are applied to the finite-dimensional representation categories of the quantized enveloping algebra $U_q(\mathfrak{g})$, the quantized Borel subalgebra $U_q^{\geq 0}(\mathfrak{g})$, and the quantized coordinate ring $\mathcal{O}_{q}(G)$ (for a suitable $q$). It is proved that whenever a tensor category is monoidally derived equivalent to the finite-dimensional representation category of any of these Hopf algebras, it is tensor equivalent to that category; see Proposition \ref{prop:quantum}.

This paper is organized as follows. We begin in Section \ref{section2} by recalling some fundamental notions and results of tensor categories and monoidal triangulated categories. In Section \ref{section3}, we introduce the definition of a simple-shifted monoidal derived equivalence and establish that a simple-shifted monoidal derived equivalence yields a tensor equivalence. Section \ref{section4} presents several sufficient conditions for the existence of simple-shifted monoidal triangulated equivalences. Finally, in Section \ref{section5}, we prove that if a tensor category is monoidally derived equivalent to the finite-dimensional representation category of either the quantized enveloping algebra $U_q(\mathfrak{g})$, the quantized Borel subalgebra $U_q^{\geq 0}(\mathfrak{g})$, or the quantized coordinate ring $\mathcal{O}_{q}(G)$ (for a suitable $q$), then it is tensor equivalent to the corresponding category.
\subsection*{Notation}
Throughout this paper, $\k$ denotes an algebraically closed field. Let $\mathbb{Z}$ denote the ring of integers. All categories we consider are essentially small, that is, the isomorphism classes of their objects form a set. Furthermore, all objects in an abelian category have finite length. All subcategories are assumed to be full. For an abelian category $\mathscr{A}$, let $\mathcal{S}_\mathscr{A}$ denote the set of isomorphism classes of simple objects in $\mathscr{A}$. For any $M\in\mathscr{A}$, we still write $M$ for the stalk complex in the bounded derived category $\mathbf{D}^{b}(\mathscr{A})$ with $M$ concentrated in degree $0$.

\section{Preliminaries}\label{section2}
In this section, we recall some basic facts concerning tensor categories and monoidal triangulated categories. For further details of monoidal categories and triangulated categories, we refer the reader to \cite{EGNO15} and \cite{Hap88}.
\subsection{Multiring categories and multitensor categories}
Let $(\mathscr{C}, \otimes, a, \mathbbm{1}_{\mathscr{C}}, l,r)$ be a monoidal category. Recall that an object $X^*$ in $\mathcal{C}$ is called a \textit{left dual} of $X$ (see \cite[Definition 2.10.1]{EGNO15}) if there exist morphisms
$$
\operatorname{ev}_X : X^* \otimes X \rightarrow \mathbbm{1}_{\mathscr{C}},
\qquad
\operatorname{coev}_X : \mathbbm{1}_{\mathscr{C}} \rightarrow X \otimes X^*,
$$
called the \textit{evaluation} and \textit{coevaluation}, respectively, such that the compositions
\begin{equation*}
X \xrightarrow{\operatorname{coev}_X \otimes \mathrm{id}_X}
(X \otimes X^*) \otimes X
\xrightarrow{a_{X,X^*,X}}
X \otimes (X^* \otimes X)
\xrightarrow{\mathrm{id}_X \otimes \operatorname{ev}_X}
X,
\end{equation*}
\begin{equation*}
X^* \xrightarrow{\mathrm{id}_{X^*} \otimes \operatorname{coev}_X}
X^* \otimes (X \otimes X^*)
\xrightarrow{a_{X^*,X,X^*}^{-1}}
(X^* \otimes X) \otimes X^*
\xrightarrow{\operatorname{ev}_X \otimes \mathrm{id}_{X^*}}
X^*
\end{equation*}
are the identity morphisms.
An object ${}^*X$ in $\mathcal{C}$ is called a \textit{right dual} of $X$ (see \cite[Definition 2.10.2]{EGNO15}) if there exist morphisms
$$
\operatorname{ev}^{\prime}_X : X \otimes {}^*X \rightarrow \mathbbm{1}_{\mathscr{C}},
\qquad
\operatorname{coev}^{\prime}_X : \mathbbm{1}_{\mathscr{C}} \rightarrow {}^*X \otimes X,
$$
such that the compositions
\begin{equation*}
X \xrightarrow{\mathrm{id}_X \otimes \operatorname{coev}^{\prime}_X}
X \otimes ({}^*X \otimes X)
\xrightarrow{a_{X,{}^*X,X}^{-1}}
(X \otimes {}^*X) \otimes X
\xrightarrow{\operatorname{ev}^{\prime}_X \otimes \mathrm{id}_X}
X,
\end{equation*}
\begin{equation*}
{}^*X \xrightarrow{\operatorname{coev}^{\prime}_X \otimes \mathrm{id}_{{}^*X}}
({}^*X \otimes X) \otimes {}^*X
\xrightarrow{a_{{}^*X,X,{}^*X}}
{}^*X \otimes (X \otimes {}^*X)
\xrightarrow{\mathrm{id}_{{}^*X} \otimes \operatorname{ev}^{\prime}_X}
{}^*X
\end{equation*}
are the identity morphisms. An object in $\mathscr{C}$ is called \textit{rigid} if it admits both a left dual and a right dual. If every object in $\mathscr{C}$ is rigid, then $\mathscr{C}$ is called \textit{rigid} (see \cite[Definition 2.10.11]{EGNO15}). By \cite[Proposition 2.10.5]{EGNO15}, if $X\in\mathscr{C}$ admits a left (respectively, right) dual, then it is unique up to a unique isomorphism. Moreover, for the left and right dual of a morphism, see \cite[Section 2.10]{EGNO15}.

\begin{definition}\emph{(}\cite[Definitions 4.1.1 and 4.2.3]{EGNO15}\emph{)}
A multiring category over $\k$ is a $\k$-linear abelian monoidal category $\mathscr{C}$ whose tensor product is bilinear and biexact. If, in addition, $\mathscr{C}$ is rigid, then $\mathscr{C}$ is called a multitensor category. We say that a multiring (respectively, multitensor) category $\mathscr{C}$ is a ring (respectively, tensor) category if $\operatorname{End}_{\mathscr{C}}(\mathbbm{1}_{\mathscr{C}})\cong \k$.
\end{definition}

Note that the condition that the morphism spaces are finite-dimensional is omitted from our definition, as it is not essential for the results of this paper.

\begin{lemma}\label{lem:reduced}
Let $\mathscr{C}$ be a ring category with left duals, i.e., every object admits a left dual. Then, for any nonzero $X,Y\in\mathscr{C}$, we have $X\otimes Y\neq 0$.
\end{lemma}
\begin{proof}
The proof is analogous to that of \cite[Proposition 3.11]{XZ26}. Indeed, we have
$$X\cong X\otimes \mathbbm {1}\hookrightarrow  X\otimes Y \otimes Y^*.$$
\end{proof}

\subsection{Monoidal triangulated categories}
\begin{definition}\emph{(} \cite[Subsection 2.2]{NVY22} \emph{)}
A monoidal triangulated category $(\mathscr{T},\otimes,[1],\mathbbm{1}_{\mathscr{T}})$ over $\k$ consists of a $\k$-linear triangulated category $(\mathscr{T},[1])$, a monoidal structure $\otimes$, and a unit object $\mathbbm{1}_{\mathscr{T}}$, where the bifunctor $\otimes$ is exact in each variable.
\end{definition}

\begin{example}\emph{(}\cite[Example 2.4]{XZ26}\emph{)}\rm\label{ex:Db(C)}
Given a multiring category $(\mathscr{C}, \otimes, \mathbbm{1}_{\mathscr{C}})$, the bounded homotopy category $\mathbf{K}^{b}(\mathscr{C})$ admits a monoidal structure. The tensor product in $\mathbf{K}^{b}(\mathscr{C})$ is defined as follows:
for any complexes $X^\bullet=(X^n, d^n_X)_{n\in\mathbb{Z}}, Y^\bullet=(Y^m, d^m_Y)_{m\in\mathbb{Z}}\in \mathbf{K}^{b}(\mathscr{C})$,  $$(X^\bullet\widetilde{\otimes} Y^\bullet)^n:=\bigoplus_{i+j=n}X^i\otimes Y^{j}$$
with differential
$$d^n_{X^{\bullet}\widetilde{\otimes} Y^{\bullet}}:=\sum\limits_{j+k=n}(d^j_{X}\otimes \id_{Y^k} +(-1)^j\id_{X^j}\otimes d_Y^k).$$
Let $[1]$ denote the shift functor of $\mathbf{K}^{b}(\mathscr{C})$.
Then $(\mathbf{K}^{b}(\mathscr{C}), \widetilde{\otimes}, [1], \mathbbm{1}_{\mathscr{C}})$ forms a monoidal triangulated category. The bounded derived category $(\mathbf{D}^{b}(\mathscr{C}), \widetilde{\otimes}, [1], \mathbbm{1}_{\mathscr{C}})$ with the inherited monoidal structure is also a monoidal triangulated category. Following \cite[Remark 2.5]{XZ26}, this bounded derived category is called a \textit{monoidal derived category}.
\end{example}

Let $\mathscr{C}$ be a multiring category, and let $(\mathbf{D}_{\mathscr{C}}^{\leq 0}, \mathbf{D}_{\mathscr{C}}^{\geq 0})$ be the standard $t$-structure on $\mathbf{D}^b(\mathscr{C})$,
where
\begin{equation*}
\begin{aligned}
\mathbf{D}_{\mathscr{C}}^{\leq 0}:=\{X\in \mathbf{D}^b(\mathscr{C})\mid H^i(X)=0,\forall i>0\},\\
\mathbf{D}_{\mathscr{C}}^{\geq 0}:=\{X\in \mathbf{D}^b(\mathscr{C})\mid H^i(X)=0,\forall i<0\}.
\end{aligned}
\end{equation*}
By \cite[Example 3.3(1)]{Big07}, $(\mathbf{D}_{\mathscr{C}}^{\leq 0}, \mathbf{D}_{\mathscr{C}}^{\geq 0})$ is a bounded $t$-structure compatible with the tensor structure (see \cite[Definition 3.2]{Big07} for the definition). We then have the following lemma.
\begin{lemma}\label{lem:Hn=HiHj}\emph{(}\cite[Theorem 4.1]{Big07}\emph{)}
Let $\mathscr{C}$ be a multiring category. Then for objects $X^\bullet, Y^\bullet\in \mathbf{D}^b(\mathscr{C})$ and an integer $n$ there is a natural isomorphism
$$
H^n(X^\bullet \widetilde{\otimes} Y^\bullet) = \coprod_{p+q=n} H^p(X^\bullet) \otimes H^q(Y^\bullet).
$$
\end{lemma}

Let $\mathscr{C}$ be a multiring category with left duals. Let $X^\bullet$ be any complex in $\mathbf{D}^b(\mathscr{C})$, written as
$$
X^\bullet:\ 0 \rightarrow X^{-n} \xrightarrow{d_X^{-n}} X^{-n+1} \xrightarrow{d_X^{-n+1}} \cdots \xrightarrow{d_X^{-1}} X^0 \xrightarrow{d_X^0} \cdots \xrightarrow{d_X^{m-2}} X^{m-1} \xrightarrow{d_X^{m-1}} X^m \rightarrow 0,
$$
where $m,n$ are positive integers. Let $Y^{-i}$ be the left dual of $X^i$, where $-n \leq i \leq m$. Then there exist evaluation and coevaluation maps denoted by
$$
\operatorname{ev}_{X^i} : Y^{-i} \otimes X^i \rightarrow \mathbbm{1}_{\mathscr{C}}, \qquad \operatorname{coev}_{X^i} : \mathbbm{1}_{\mathscr{C}} \rightarrow X^i \otimes Y^{-i}.
$$
Inspired by \cite[Example 2.27]{XZ26}, the complex
$$
Y^\bullet : 0 \rightarrow Y^{-m} \xrightarrow{(-1)^{m+2}(d_X^{m-1})^*} Y^{-m+1} \rightarrow \cdots \xrightarrow{(-1)^n(d_X^{-n+1})^*} Y^{n-1} \xrightarrow{(-1)^{n+1}(d_X^{-n})^*} Y^n \rightarrow 0
$$
is the left dual of $X^\bullet$, with evaluation
$$
\widetilde{\operatorname{ev}}_{X^\bullet} := \bigoplus_{i=-n}^m \operatorname{ev}_{X^i}
$$
and coevaluation
$$
\widetilde{\operatorname{coev}}_{X^\bullet} := \bigoplus_{i=-n}^m \operatorname{coev}_{X^i}.
$$
If $\mathscr{C}$ also has right duals, the right dual of $X^\bullet$ can be constructed similarly. By \cite[Proposition 4.2.9]{EGNO15}, the left dualization functor in $\mathscr{C}$ is exact. In other words, if $0\to M \to Z\to N\to 0$ is a short exact sequence in $\mathscr{C}$, then so is $0\to N^*\to Z^*\to M^*\to 0$. It follows that $H^{-i}(Y^\bullet)\cong (H^{i}(X^\bullet))^* $ for all $i\in\mathbb{Z}.$

For every nonzero object $X^\bullet \in D^b(\mathscr{C})$, define
$$
a(X^\bullet) = \min\{i : H^i(X^\bullet) \neq 0\}, \qquad b(X^\bullet) = \max\{i : H^i(X^\bullet) \neq 0\}.
$$
\begin{corollary}\label{coro:a,b}
Let $\mathscr{C}$ be a ring category with left duals.
Then we have $a(X^\bullet\widetilde{\otimes} Y^\bullet)=a(X^\bullet)+a(Y^{\bullet})$ and $b(X^\bullet\widetilde{\otimes} Y^\bullet)=b(X^\bullet)+b(Y^{\bullet})$ for all nonzero objects $X^\bullet, Y^\bullet \in \mathbf{D}^b(\mathscr{C})$.
\end{corollary}
\begin{proof}
This is an immediate consequence of Lemmas \ref{lem:reduced} and \ref{lem:Hn=HiHj}.
\end{proof}

A \textit{monoidal triangulated functor} is a $\k$-linear functor that is both triangulated and monoidal (as defined in \cite[Definition 2.4.1]{EGNO15}). A monoidal triangulated functor is said to be an \textit{equivalence} if it is an equivalence of ordinary categories.
Let $\mathscr{C}$ and $\mathscr{D}$ be multiring categories, and let $\mathbf{D}^{b}(\mathscr{C})$ and $\mathbf{D}^{b}(\mathscr{D})$ be monoidal derived categories as in Example \ref{ex:Db(C)}.
Following \cite[Remark 2.5]{XZ26}, monoidal triangulated equivalences between $\mathbf{D}^{b}(\mathscr{C})$ and $\mathbf{D}^{b}(\mathscr{D})$ will simply be called \textit{monoidal derived equivalences}.

\section{Tensor equivalences from simple-shifted monoidal derived equivalences}\label{section3}
In this section, we first introduce the definition of a simple-shifted equivalence between bounded derived categories, and prove that every simple-shifted monoidal derived equivalence between monoidal derived categories gives rise to a tensor equivalence in the sense of \cite[Definition 4.2.5]{EGNO15}.

\begin{definition}\label{def:simple-shifted}
Let $\mathscr{A}$ and $\mathscr{B}$ be abelian categories. A simple-shifted functor from $\mathbf{D}^{b}(\mathscr{A})$ to $\mathbf{D}^{b}(\mathscr{B})$ is a triple $(F,\psi,r)$ consisting of a triangulated functor $F:\mathbf{D}^{b}(\mathscr{A}) \rightarrow \mathbf{D}^{b}(\mathscr{B})$, a set map $\psi:\mathcal{S}_{\mathscr{A}}\rightarrow \mathcal{S}_{\mathscr{B}}$, and a set map $r:\mathcal{S}_{\mathscr{A}}\rightarrow \mathbb{Z}$, subject to the condition $F(S)\cong \psi(S)[r(S)]$ for all $S\in \mathcal{S}_{\mathscr{A}}$. We call $(F,\psi,r)$ an equivalence if $F$ is an equivalence of triangulated categories.
\end{definition}

In what follows, let $\mathscr{A}$ and $\mathscr{B}$ be abelian categories, and let $(F,\psi,r)$ be a simple-shifted equivalence from $\mathbf{D}^{b}(\mathscr{A})$ to $\mathbf{D}^{b}(\mathscr{B})$.
The following lemma shows that the equivalence of $(F,\psi,r)$ implies the bijectivity of $\psi$.
\begin{lemma}\label{lem:psibijective}
We have that $\psi:\mathcal{S}_{\mathscr{A}}\rightarrow \mathcal{S}_{\mathscr{B}}$ is bijective.
\end{lemma}
\begin{proof}
Suppose that $\psi(S)=\psi(T)$ for some $S,T\in\mathcal{S}_{\mathscr{A}}$. Then we have
\begin{equation*}
F(T[r(S)-r(T)]) \cong F(T)   [r(S)-r(T)]\cong \psi(T)[r(S)] \cong F(S).
\end{equation*}
Since $F$ is an equivalence, it follows that $$T[r(S)-r(T)]\cong S$$ in $\mathbf{D}^{b}(\mathscr{A})$. This means that $r(S)=r(T)$ and thus $\psi$ is injective. We define $$\mathcal{R}=\{\psi (S)\mid S\in \mathcal{S}_{\mathscr{A}}\}.$$ Suppose, on the contrary, that $\psi$ is not surjective; that is, there exists some $U\in\mathcal{S}_{\mathscr{B}}$ such that
$
U\notin \mathcal{R}.
$ By the density of $F$, there exists $X^\bullet\in \mathbf{D}^{b}(\mathscr{A})$ such that $F(X^\bullet)\cong U.$ Let $$0\rightarrow S_1\rightarrow M \rightarrow S_2\rightarrow 0$$ be a short exact sequence in $\mathscr{A}$, where $S_1, S_2\in\mathcal{S}_{\mathscr{A}}$. Then  $$ S_1\rightarrow M \rightarrow S_2\rightarrow S_1[1]$$ is a distinguished triangle in $\mathbf{D}^b(\mathscr{A})$. We have a distinguished triangle $$ F(S_1)\rightarrow F(M) \rightarrow F(S_2)\rightarrow F(S_1)[1]$$ in $\mathbf{D}^b(\mathscr{B})$. The long exact sequence of cohomology objects corresponding to the latter distinguished triangle reads as follows:
$$
\cdots\rightarrow H^n(F(S_1))\rightarrow H^n(F(M))\rightarrow H^n(F(S_2))\rightarrow  H^{n+1}(F(S_1))\rightarrow\cdots.
$$
Hence, for every $n\in\mathbb{Z}$ with $H^n(F(M))\neq0$, the composition factors of $H^n(F(M))$ lie in the set $\mathcal{R}$. By induction on the length of objects, one can prove that, for every $N\in\mathscr{A}$ and every $n\in\mathbb{Z}$ with $H^n(F(N))\neq0$, the composition factors of $H^n(F(N))$ lie in the set $\mathcal{R}$. We may consider the following distinguished triangle:
$$
\tau_{\leq m-1} X^{\bullet} \rightarrow X^{\bullet} \rightarrow H^m(X^{\bullet})[-m] \rightarrow \tau_{\leq m-1}X^{\bullet}[1],
$$
where $m=b(X^\bullet).$
Then we have the following distinguished triangle:
$$
F(\tau_{\leq m-1} X^{\bullet}) \rightarrow F(X^{\bullet}) \rightarrow F(H^m(X^{\bullet})[-m]) \rightarrow F(\tau_{\leq m-1}X^{\bullet})[1].
$$
The long exact sequence of cohomology objects corresponding to the above distinguished triangle shows that the composition factors of $H^n(F(X^\bullet))$ arise from those of $H^n(F(\tau_{\leq m-1}X^\bullet))$ and $H^n(F(H^m(X^\bullet)[-m]))$. Note that for every $n\in\mathbb{Z}$ with $H^n(F(H^m(X^{\bullet})[-m]))\neq0$, the composition factors of $H^n(F(H^m(X^{\bullet})[-m]))$ lie in the set $\mathcal{R}$.
By induction, we can prove that, for every $n\in\mathbb{Z}$ with $H^n(F(X^{\bullet}))\neq0$, the composition factors of $H^n(F(X^{\bullet}))$ lie in the set $\mathcal{R}$. This is a contradiction. Thus, $\psi$ is bijective.
\end{proof}

Since Lemma \ref{lem:psibijective} implies that $\psi$ is bijective, we may define
\begin{equation}
s: \mathcal{S}_{\mathscr{B}}   \rightarrow \mathbb{Z},\quad U\mapsto r(\psi^{-1}(U)).
\end{equation}
It follows that $(F^{-1}, \psi^{-1}, -s)$ is also a simple-shifted derived equivalence.
Let $M,N\in\mathscr{A}$. By \cite[Proposition 3.2.2]{Ver96}
, we have
\begin{equation}\label{eq:ext=homDb}
\operatorname{Ext}_{\mathscr{A}}^k(M, N)   \cong \operatorname{Hom}_{\mathbf{D}^b(\mathscr{A})}(M,N[k]), \quad k\geq0.
\end{equation}
With the aid of this formula, we can prove the following lemma.

\begin{lemma}\label{lem:Ext1=Exti}
If $\operatorname{Ext}_{\mathcal{B}}^{1}(U,V)=0$ holds for all $U,V\in\mathcal{S}_{\mathcal{B}}$ with $s(U)<s(V)$, then $\operatorname{Ext}_{\mathcal{B}}^{k}(U,V)=0$ holds for all $U,V\in\mathcal{S}_{\mathcal{B}}$ with $s(U)<s(V)$ and for all $k\geq 1$.
\end{lemma}

\begin{proof}
Let $\mathcal{B}^{\leq s(U)}$ and $\mathcal{B}^{> s(U)}$ be the full subcategories of $\mathcal{B}$ defined by
\begin{equation*}
\begin{aligned}
\mathcal{B}^{\leq s(U)}:=\{X\in\mathcal{B}\mid s(W)\leq s(U)\text{ for every composition factor }W\text{ of }X\},\\
\mathcal{B}^{> s(U)}:=\{X\in\mathcal{B}\mid s(W)>s(U)\text{ for every composition factor }W\text{ of }X\}.
\end{aligned}
\end{equation*}
Clearly, $\mathcal{B}^{\leq s(U)}$ and $\mathcal{B}^{> s(U)}$ are Serre subcategory, and we have $$\operatorname{Hom}_{_{\mathcal{B}}}(\mathcal{B}^{\leq s(U)}, \mathcal{B}^{> s(U)})=0.$$
We claim that $\operatorname{Ext}^{1}(M,N)=0$ for all $M\in \mathcal{B}^{\leq s(U)}$ and all $N\in \mathcal{B}^{> s(U)}$. Indeed, this can be shown by induction on the length of the objects on both sides separately, and then by applying \cite[Theorem 3.1]{Buc59}. According to \cite[Lemma 2.7]{BB24}, $\mathcal{B}^{\leq s(U)}$ is right admissible in the sense of \cite[Definition 2.5]{BB24} and $(\mathcal{B}^{\leq s(U)})^\bot = \mathcal{B}^{> s(U)},$
where
$$
( \mathcal{B}^{\leq s(U)})^\bot:=\{X\in \mathscr{B}\mid \operatorname{Hom}_{\mathscr{B}}(\mathcal{B}^{\leq s(U)},X)=0=\operatorname{Ext}^1_{\mathscr{B}}(\mathcal{B}^{\leq s(U)},X)\}.
$$
Then by \cite[Proposition 2.6]{BB24}, $\mathcal{B}^{> s(U)}$ is left admissible in the sense of \cite[Definition 2.5]{BB24}. By \cite[Theorem 2.25]{BB24}, $\mathbf{D}^b(\mathcal{B}^{\leq s(U)})$ is left orthogonal to $\mathbf{D}^b(\mathcal{B}^{> s(U)})$ (for the definition, see \cite[Subsection 3]{Bon89}
).
It follows from (\ref{eq:ext=homDb}) that $$\operatorname{Ext}^{k}_{\mathscr{B}}(U,V)\cong \operatorname{Hom}_{\mathbf{D}^b(\mathscr{B})}(U, V[k])=0$$ for all $k\geq 0.$
\end{proof}
Note that the above proof actually remains valid if $s$ is replaced by an arbitrary map $t:\mathcal{S}_\mathscr{B}\rightarrow\mathbb{Z}$.

\begin{lemma}\label{lem:ext1=0}
For any $S,T\in\mathcal{S}_{\mathscr{A}}$, if $r(S)\neq r(T)$, then $\operatorname{Ext}_{\mathscr{A}}^1(S,T)=0.$ Moreover, for any $U, V\in \mathcal{S}_\mathscr{B}$, if $s(U)\neq s(V)$, then $\operatorname{Ext}_{\mathscr{B}}^1(U,V)=0.$
\end{lemma}
\begin{proof}
By (\ref{eq:ext=homDb}), we have
\begin{equation*}
    \begin{aligned}
\operatorname{Ext}_{\mathscr{A}}^1(S,T)
&\cong \operatorname{Hom}_{\mathbf{D}^b(\mathscr{A})}(S, T[1])\\
&\cong \operatorname{Hom}_{\mathbf{D}^b(\mathscr{B})}(F(S), F(T)[1])\\
&\cong \operatorname{Hom}_{\mathbf{D}^b(\mathscr{B})} (\psi(S), \psi(T)[1+r(T)-r(S)]).
    \end{aligned}
\end{equation*}
If $1+r(T)-r(S)=0$, then we have
\begin{equation*}
\operatorname{Hom}_{\mathbf{D}^b(\mathscr{B})} (\psi(S), \psi(T))   \cong \operatorname{Hom}_{\mathscr{B}}(\psi(S), \psi(T)).
\end{equation*}
By Lemma \ref{lem:psibijective}, $\psi$ is bijective, so we have $\psi(S)\ncong \psi(T)$, and hence $$\operatorname{Ext}^1_{\mathscr{A}}(S,T)= 0.$$
If $1+r(T)-r(S)\leq -1$, then, by the fact that $(\mathbf{D}^{\leq 0}_{\mathscr{B}},\mathbf{D}^{\geq 0}_{\mathscr{B}})$ constitutes a standard t-structure, we conclude that $$\operatorname{Ext}^1_{\mathscr{A}}(S,T)=0.$$
Next, we deal with the case $1+r(T)-r(S)\geq 2$. In this case, (\ref{eq:ext=homDb}) implies that
\begin{equation*}
    \begin{aligned}
\operatorname{Hom}_{\mathbf{D}^b(\mathscr{B})} (\psi(S), \psi(T)[1+r(T)-r(S)])
\cong \operatorname{Ext}_{\mathscr{B}}^{1+r(T)-r(S)}(\psi(S), \psi(T)).
    \end{aligned}
\end{equation*}
Since $1-r(T)+r(S)\leq 0$ and $(F^{-1}, \psi^{-1}, -s)$ is also a simple-shifted derived equivalence, a similar argument shows that
$$
\operatorname{Ext}_{\mathscr{B}}^{1}(\psi(S), \psi(T))\cong \operatorname{Hom}_{\mathbf{D}^b(\mathscr{A})} (S, T[1-r(T)+r(S)])=0.
$$
By Lemma \ref{lem:Ext1=Exti}, we have
$$
\operatorname{Ext}^1_{\mathscr{A}}(S,T)\cong \operatorname{Ext}_{\mathscr{B}}^{1+r(T)-r(S)}(\psi(S), \psi(T))=0.
$$
Using the fact that $(F^{-1}, \psi^{-1}, -s)$ is also a simple-shifted derived equivalence, we can obtain the remaining conclusions.
\end{proof}

Let $I:=\{r(S)\mid S\in\mathcal{S}_{\mathscr{A}}\}$.
For any $i\in I$, define $\mathscr{A}_i$ and $\mathscr{B}_i$ to be the full subcategories of $\mathscr{A}$ and $\mathscr{B}$, respectively, by
\begin{equation*}
\begin{aligned}
\mathscr{A}_i:=\{X\in\mathcal{A}\mid r(S)=i\text{ for every composition factor }S\text{ of }X\},\\
\mathscr{B}_i:=\{X\in\mathcal{B}\mid s(U)=i\text{ for every composition factor }U\text{ of }X\}.
\end{aligned}
\end{equation*}
For every $i\in I$, we know that $\mathscr{A}_i$ and $\mathscr{B}_i$ are Serre subcategory of $\mathscr{A}$ and $\mathscr{B}$, respectively. Lemma \ref{lem:ext1=0} implies that for any $S,T\in\mathcal{S}_{\mathscr{A}}$ with $\operatorname{Ext}_{\mathscr{A}}^1(S,T)\neq0$, we have $r(S)=r(T)$. This implies that $r$ is constant on any two simple objects of $\mathscr{A}$ that are linked in the sense of \cite[Exercise 1.5.10(ii)]{EGNO15}.
By \cite[Exercise 1.5.10(i)]{EGNO15}, there is a unique decomposition
$$
\mathscr{A}\cong \bigoplus_{\gamma\in \Gamma} \mathscr{A}^{\gamma},
$$
such that $\mathscr{A}^{\gamma}$ is indecomposable (see \cite[Definition 1.3.7]{EGNO15} for the definition). The categories $\mathscr{A}^{\gamma}$ are called the \textit{blocks} of $\mathscr{A}.$
For every $\gamma \in \Gamma$, \cite[Exercise 1.5.10(ii)]{EGNO15} implies that there exists some $i_{\gamma}\in I$ such that $\mathscr{A}^{\gamma}\subseteq \mathscr{A}_{i_\gamma}$.
It follows that
$$
\mathscr{A}\cong \bigoplus_{i\in I}\mathscr{A}_i,\quad \mathscr{B}\cong \bigoplus_{i\in I}\mathscr{B}_i.
$$

\begin{lemma}\label{lem:F(A_i)}
For any $i\in I$, let $\mathscr{B}_i[i]$ be the full subcategory of $\mathbf{D}^b(\mathscr{B})$ defined by
$$
\mathscr{B}_i[i]:=\{X^\bullet\in\mathbf{D}^b(\mathscr{B})\mid X^\bullet\cong W[i] \text{ for some }W\in \mathscr{B}_i\}.
$$
Then we have $F(\mathscr{A}_i)\subseteq \mathscr{B}_i[i].$
\end{lemma}
\begin{proof}
We prove this by induction on the length of objects in $\mathscr{A}_i$. Clearly, the assertion holds for simple objects. Assume that it holds for objects of length at most $n-1$. For an object $X$ of length $n$, we take a short exact sequence
$$
0 \rightarrow Y \rightarrow X \rightarrow S \rightarrow 0,
$$
where $S$ is a simple object. The induction hypothesis implies that there exists $Y^\prime \in \mathscr{B}_i$ such that
$F(Y)\cong Y^\prime[i]$ in $\mathbf{D}^b(\mathscr{B})$.
Consider the following distinguished triangle
\begin{equation*}
Y \rightarrow X \rightarrow S \rightarrow Y[1].
\end{equation*}
Applying $ F$ to the above distinguished triangle, we obtain
\begin{equation*}
Y^\prime[i] \rightarrow F(X) \rightarrow \psi(S)[i] \rightarrow Y^\prime[i+1].
\end{equation*}
The long exact sequence of cohomology objects associated with this distinguished triangle is as follows:
\begin{equation*}
\cdots\rightarrow 0\rightarrow Y^\prime \rightarrow H^{-i}(F(X)) \rightarrow \psi(S) \rightarrow 0 \rightarrow \cdots   .
\end{equation*}
It follows that $H^{-i}(F(X))\in \mathscr{B}_i,$ and that $F(X)\cong H^{-i}(F(X))[i].$
\end{proof}

Since $F(\mathscr{A}_i)\subseteq \mathscr{B}_i[i]$, for every $i\in I$, we may define a functor as follows:
\begin{equation*}
    G_i:\mathscr{A}_i\rightarrow \mathscr{B}_i,\quad X\mapsto H^{-i}(F(X)).
\end{equation*}
We then define $G:=\bigoplus_{i\in I}G_i$.

\begin{proposition}\label{prop:equ}
The functor $G$ is an equivalence between abelian categories $\mathscr{A}$ and $\mathscr{B}$.
\end{proposition}
\begin{proof}
Taking a short exact sequence $$0\rightarrow Y\rightarrow X\rightarrow Z\rightarrow 0$$ in $\mathscr{A}_i$, we obtain the distinguished triangle
$$
F(Y)\rightarrow F(X)\rightarrow F(Z)\rightarrow F(Y)[1].
$$
By Lemma \ref{lem:F(A_i)}, we obtain the following long exact sequence of cohomology objects:
$$
\cdots\rightarrow 0\rightarrow H^{-i}(F(Y))\rightarrow H^{-i}(F(X))\rightarrow H^{-i}(F(Z))\rightarrow 0 \rightarrow \cdots.
$$
It follows that $G_i$ is exact, and hence $G$ is exact.
Note that the embedding $\mathscr{B}\hookrightarrow \mathbf{D}^b(\mathscr{B})$ is fully faithful. For any $M, N\in \mathscr{A}_i$, we have
\begin{equation*}
    \begin{aligned}
       \operatorname{Hom}_{\mathscr{B}_i}(G_i(M), G_i(N))&\cong \operatorname{Hom}_{\mathbf{D}^b(\mathscr{B})}(G_i(M), G_i(N)) \\
       &\cong \operatorname{Hom}_{\mathbf{D}^b(\mathscr{B})}(H^{-i}(F(M))[i], H^{-i}(F(N))[i])\\
       &\cong \operatorname{Hom}_{\mathbf{D}^b(\mathscr{B})}(F(M), F(N))\\
       &\cong \operatorname{Hom}_{\mathbf{D}^b(\mathscr{A})}(M, N)\\
       &\cong \operatorname{Hom}_{\mathscr{A}}(M, N).
    \end{aligned}
\end{equation*}
Consequently, $G_i$ is fully faithful, and therefore so is $G$. By an argument similar to the proof of Lemma \ref{lem:F(A_i)}, we can prove that for every $M\in \mathscr{B}_i$, there exists some $N\in \mathscr{A}_i$ such that $F^{-1} (M) \cong N[-i].$ Thus, in $\mathbf{D}^b(\mathscr{B})$ we have
$$
M\cong F(F^{-1} (M))\cong F(N[-i])\cong G_i(N).
$$
We conclude that $G_i$ is dense, from which it follows that $G$ is dense.
\end{proof}

In what follows, we let $\mathscr{A}$ and $\mathscr{B}$ be multiring categories over $\k$. To avoid confusion with the abelian case, we relabel them as $\mathscr{C}$ and $\mathscr{D}$, respectively. Suppose that $(F,\psi,r)$ is a simple-shifted monoidal derived equivalence between $\mathbf{D}^b(\mathscr{C})$ and $\mathbf{D}^b(\mathscr{D})$. In other words, $(F,\psi,r)$ is a simple-shifted equivalence such that $(F,J)$ is a monoidal equivalence between monoidal derived categories, where $J_{X^\bullet,Y^\bullet}:F(X^\bullet)\widetilde{\otimes} F(Y^\bullet)\rightarrow F(X^\bullet\widetilde{\otimes} Y^\bullet)$ is a natural isomorphism for all $X^\bullet,Y^\bullet\in \mathbf{D}^b(\mathscr{C})$. See \cite[Definition 2.4.1]{EGNO15} for the definition of a monoidal functor.

\begin{lemma}\label{lem:CiCj<Cij}
For any $i,j\in I$, we have $\mathscr{C}_i\otimes \mathscr{C}_j\subseteq \mathscr{C}_{i+j}$ and $\mathscr{D}_i\otimes \mathscr{D}_j\subseteq \mathscr{D}_{i+j}$. In particular, $\mathbbm{1}_{\mathscr{C}}\in \mathscr{C}_0$ and $\mathbbm{1}_{\mathscr{D}}\in \mathscr{D}_0$.
\end{lemma}
\begin{proof}
By Lemma \ref{lem:F(A_i)}, for any $X\in\mathscr{C}_i$, $Y\in\mathscr{C}_j$, we have
$$
F(X\otimes Y)\cong F(X)\widetilde{\otimes} F(Y)\cong (H^{-i}F(X)\otimes H^{-j}F(Y))[i+j].
$$
Suppose that $X\otimes Y \cong \bigoplus_{k\in I} Z_k$ with $Z_k\in \mathscr{C}_k$. Then we have
$$
F(X\otimes Y)\cong \bigoplus_{k\in I} F(Z_k)\cong \bigoplus_{k\in I} H^{-k}(F(Z_k))[k].
$$
It follows that $X\otimes Y \in \mathscr{C}_{i+j}$. A similar argument shows that $\mathscr{D}_i\otimes \mathscr{D}_j\subseteq \mathscr{D}_{i+j}$.
\end{proof}

\begin{proposition}\label{prop:tensorequ}
The functor $G$ is a tensor equivalence between multiring categories $\mathscr{C}$ and $\mathscr{D}$.
\end{proposition}

\begin{proof}
By Proposition \ref{prop:equ}, the functor $G$ is an exact equivalence.
For any $X\in\mathscr{C}_i$, $Y\in\mathscr{C}_j$, and $n\in\mathbb{Z}$, by Lemma \ref{lem:Hn=HiHj}, there is a natural isomorphism
$$
H^n(F(X)\widetilde{\otimes} F(Y))\cong \bigoplus_{p+q=n} H^p(F(X))\otimes H^q(F(Y)).
$$
It follows that
$$
 H^{-i-j}(F(X)\widetilde{\otimes} F(Y))\cong H^{-i} (F(X))\otimes H^{-j} (F(Y)) .
$$
Then we have
\begin{equation*}
\begin{aligned}
G_i(X)\otimes G_j(Y) &=H^{-i} (F(X))\otimes H^{-j} (F(Y))\\ &\cong H^{-i-j}(F(X)\widetilde{\otimes} F(Y))\\
&\cong H^{-i-j}(F(X\otimes Y))\\
&= G_{i+j}(X\otimes Y).
\end{aligned}
\end{equation*}
Hence, one may directly prove that $G$ is a tensor equivalence.
\end{proof}

Note that if $\mathscr{C}$ is a multitensor category with left duals, then, by \cite[Proposition 4.2.9]{EGNO15}, the left dualization functor is exact. This implies that if $S\in \mathcal{S}_{\mathscr{C}}$, then $S^*\in \mathcal{S}_{\mathscr{C}}$. Moreover, by \cite[Exercise 2.10.6]{EGNO15}, we have
$$
F(S^*)\cong F(S)^*\cong (\psi(S)[r(S)])^*\cong \psi(S)^*[-r(S)].
$$
It follows that $r(S^*)=-r(S).$ Clearly, $\mathscr{C}_i^* \approx \mathscr{C}_{-i}$ hold for all $i\in I.$

Recall from \cite[Definition 4.11.1]{EGNO15}, a \textit{tensor subcategory} in a tensor category is an additive subcategory which is closed under taking subquotients, tensor products and duality.
We obtain the following corollary.

\begin{corollary}
Let $\mathscr{C}$ and $\mathscr{D}$ be tensor categories, and let $(F,\psi,r)$ be a simple shifted monoidal derived equivalence between $\mathbf{D}^b(\mathscr{C})$ and $\mathbf{D}^b(\mathscr{D})$. For each $i\in I$, define $\mathscr{E}_i= \sum_{m\in\mathbb{Z}} \mathscr{C}_{im}$ and $\mathscr{F}_i=\sum_{m\in\mathbb{Z}} \mathscr{D}_{im}$. Then $\mathscr{E}_i$ and $\mathscr{F}_i$ are tensor subcategories of $\mathscr{C}$ and $\mathscr{D}$, respectively. Moreover, $G\mid_{\mathscr{E}_i}:\mathscr{E}_i\rightarrow \mathscr{F}_i$ is a tensor equivalence.
\end{corollary}

\section{The existence of a simple-shifted derived monoidal equivalence}\label{section4}
In this section, we consider under what conditions a simple-shifted monoidal derived equivalence exists between monoidal derived categories. For simplicity, if a monoidal triangulated functor $F$ between monoidal derived categories admits a triple that is simple-shifted, then we simply say that $F$ is simple-shifted. By the Mac Lane strictness theorem (\cite[Theorem 2.8.5]{EGNO15}), every monoidal category is monoidal equivalent to a strict monoidal category. For simplicity, the associator is omitted throughout this section.

Let $\mathscr{C}$ be a ring category with left duals.
Recall that an object $X\in \mathscr{C}$ is called \textit{invertible} if $\operatorname{ev}_X$ and $\operatorname{coev}_X$ are isomorphisms.
According to \cite[Exercise 4.3.11(3)]{EGNO15}, every invertible object is simple.
We claim that if $X\in\mathscr{C}$ is invertible and $Y\in\mathscr{C}$ is a simple object, then $X\otimes Y$ is a simple object. Indeed, if there exists a simple object $Z\in \mathscr{C}$ such that $Z\hookrightarrow X\otimes Y$, then we have $X^*\otimes Z\hookrightarrow X^*\otimes X\otimes Y$. Since $X^*\otimes X\otimes Y\cong Y$, it follows that $Z\cong X\otimes Y.$ The same assertion also holds for $Y\otimes X$. Moreover, by \cite[Proposition 2.11.3(i)]{EGNO15}, we have ${}^*X\cong X^*$ and $X^*$ is invertible.

\begin{lemma}\label{lem:invertible}
Let $\mathscr{C}$ and $\mathscr{D}$ be ring categories with left duals, and let $S$ be an invertible object in $\mathscr{C}$. If $F:\mathbf{D}^b(\mathscr{C})\rightarrow \mathbf{D}^b(\mathscr{D})$ is a monoidal derived equivalence, then $F(S)\cong T[n]$ for some invertible object $T\in\mathscr{D}$ and $n\in\mathbb{Z}$.
\end{lemma}

\begin{proof}
We have
\begin{equation*}
F(S^*)\widetilde{\otimes} F(S)\cong  F(S^*\otimes S)\cong F(\mathbbm{1}_{\mathscr{C}}) \cong \mathbbm{1}_{\mathscr{D}}.
\end{equation*}
Then, by Corollary \ref{coro:a,b}, we obtain
$$a(F(S^*))+a(F(S))=b(F(S^*))+b(F(S))=0.$$
It follows that
$$b(F(S))-a(F(S))=b(F(S^*))-a(F(S^*))=0,$$
and hence
$$b(F(S^*))=a(F(S^*))=-b(F(S))=-a(F(S)).$$
Thus, we have $$F(S)\cong H^{a(F(S))}(F(S))[-a(F(S))]$$ and $$F(S^*)\cong H^{-a(F(S))}(F(S^*))[a(F(S))].$$ From \cite[Exercise 2.10.6]{EGNO15}, we can directly prove that $H^{a(F(S))}(F(S))$ is invertible.
\end{proof}

Recall from \cite[Definition 5.11.1]{EGNO15} that a tensor category $\mathscr{C}$ is called \textit{pointed} if every simple object of $\mathscr{C}$ is invertible. It is well known that the finite-dimensional comodule category of a pointed Hopf algebra gives rise to a pointed tensor category, where a Hopf algebra is called \textit{pointed} if every simple subcoalgebra is one-dimensional. By virtue of Lemma \ref{lem:invertible}, we can prove that every monoidal derived equivalence between pointed tensor categories implies a tensor equivalence, and that such a monoidal derived equivalence is exactly simple-shifted.
\begin{proposition}\label{prop:pointed}
Let $\mathscr{C}$ and $\mathscr{D}$ be tensor categories, and let $F: \mathbf{D}^b(\mathscr{C})\rightarrow \mathbf{D}^b(\mathscr{D})$ be a monoidal derived equivalence. Suppose that $\mathscr{C}$ is pointed. Then $F$ is a simple-shifted monoidal derived equivalence. Moreover, $\mathscr{C}$ is tensor equivalent to $\mathscr{D}.$
\end{proposition}
\begin{proof}
From the proof of Lemma \ref{lem:invertible}, for every simple object $S\in\mathscr{C}$,
$$
F(S)\cong H^{a(F(S))}(F(S))[-a(F(S))],
$$
with $H^{a(F(S))}(F(S))$ invertible. According to \cite[Exercise 4.3.11(3)]{EGNO15}, $H^{a(F(S))}(F(S))$ is simple. Then we can define maps
\begin{equation*}
\begin{aligned}
&\psi:\mathcal{S}_{\mathscr{C}}\rightarrow \mathcal{S}_{\mathscr{D}},\qquad S\mapsto H^{a(F(S))}(F(S)),\\
&r:\mathcal{S}_{\mathscr{C}}\rightarrow \mathbb{Z},\qquad S\mapsto -a(F(S)).
\end{aligned}
\end{equation*}
Clearly, $(F,\psi, r)$ is simple-shifted. Then by Proposition \ref{prop:tensorequ}, $\mathscr{C}$ and $\mathscr{D}$ are tensor equivalent.
\end{proof}

Before proceeding further, we state the following lemma.
\begin{lemma}\label{lem:finite}
Let $\mathscr{C}$ and $\mathscr{D}$ be tensor categories, and let $F: \mathbf{D}^b(\mathscr{C})\rightarrow \mathbf{D}^b(\mathscr{D})$ be a monoidal triangulated functor. Suppose that $\mathscr{C}$ has finitely many simple objects. Then $a(F(X))=b(F(X))=0$ for all $X\in\mathscr{C}$ with $F(X)\neq 0$.
\end{lemma}
\begin{proof}
The proof is analogous to that of \cite[Theorem 4.5]{XZ26}; for the sake of completeness, however, we provide an alternative proof. Assume that $\mathcal{S}_{\mathscr{C}}=\{S_i\mid 1\leq i\leq n\}$. Then there exist integers $u$ and $v$ such that
$$
u\leq a(F(S_i))\leq b(F(S_i))\leq v
$$
for all $1\leq i\leq n$ with $F(S_i)\neq 0$. By induction on the length of objects in $\mathscr{C}$, one can prove that
$$
u\leq a(F(X))\leq b(F(X))\leq v
$$
 for all $X\in\mathscr{C}$ with $F(X)\neq 0$. Note that for every $1\leq i\leq n$ and any positive integer $m$, we have
$$
u\leq a(F(S_i^{\otimes m}))\leq b(F(S_i^{\otimes m}))\leq v.
$$
By Corollary \ref{coro:a,b}, we obtain $a(F(X))=b(F(X))=0$.
\end{proof}

In a tensor category $\mathscr{C}$, we may consider the tensor subcategory $\mathscr{E}$ generated by the unit object $\mathbbm{1}_{\mathscr{C}}$ (which is isomorphic to $\mathrm{Vec}_{\k}$). It is the smallest tensor subcategory containing $\mathbbm{1}_{\mathscr{C}}$. Clearly, it has a unique simple object, namely $\mathbbm{1}_{\mathscr{C}}$. In a pointed tensor category, for any simple object $X$, $X\otimes X^*$ and $X^*\otimes X$ lie in $\mathscr{E}$. For Proposition \ref{prop:pointed}, we can obtain a more general version.
\begin{theorem}\label{thm:SotimesS*}
Let $\mathscr{C}$ and $\mathscr{D}$ be tensor categories, and let $F: \mathbf{D}^b(\mathscr{C})\rightarrow \mathbf{D}^b(\mathscr{D})$ be a monoidal derived equivalence. Suppose that for every simple object $S\in \mathscr{C}$, there exists a tensor subcategory $\mathscr{E}_{S}$ of $\mathscr{C}$ with finitely many simple objects such that either $S\otimes S^* \in \mathscr{E}_{S}$ or $S^*\otimes S\in \mathscr{E}_{S}$. Then $F$ is simple-shifted. Moreover, $\mathscr{C}$ is tensor equivalent to $\mathscr{D}.$
\end{theorem}
\begin{proof}
For any simple object $S\in\mathscr{C}$, we may assume without loss of generality that there exists a tensor subcategory $\mathscr{E}_{S}$ of $\mathscr{C}$ with finitely many simple objects such that $S\otimes S^*\in \mathscr{E}_{S}$.
Note that we have
$$
F(S)\widetilde{\otimes }F(S^*)\cong F(S\otimes S^*).
$$
From Lemma \ref{lem:finite}, we obtain that $$a(F(S)\widetilde{\otimes }F(S^*))=b(F(S)\widetilde{\otimes }F(S^*))=0.$$
A similar argument as in the proof of Lemma \ref{lem:invertible} shows that
$$F(S)\cong H^{a(F(S))}(F(S))[-a(F(S))]$$ and $$F(S^*)\cong H^{-a(F(S))}(F(S^*))[a(F(S))].$$
For simplicity, we write $H^{a(F(S))}(F(S))$ by $X.$ It follows that $$H^0 F(S\otimes S^*)\cong X\otimes X^*.$$
We define a subset of $\mathcal{S}_{\mathscr{D}}$ and a full subcategory of $\mathscr{D}$ as follows:
$$
\mathcal{R}:=\{R\in\mathcal{S}_{\mathscr{D}}\mid R \text{ is a composition factor of } H^0F(T), \text{ where } T\in \mathcal{S}_{\mathscr{E}_S} \},
$$
$$
\mathscr{F}_{\mathcal{R}}:=\{M\in \mathscr{D}\mid \text{the composition factors of } M \text{ arise from } \mathcal{R}\}.
$$
By Lemma \ref{lem:finite} and the argument in the proof of Proposition \ref{prop:equ}, the functor $H^0 F\mid_{\mathscr{E}_{S}}$ is exact.
By induction on the length of objects in $\mathscr{E}_{S}$, we can prove that
$
H^0 F(\mathscr{E}_{S}) \subseteq \mathscr{F}_{\mathcal{R}}.
$
We claim that $\mathscr{F}_{\mathcal{R}}$ is a tensor subcategory of $\mathscr{D}$ with finitely many simple objects.
Indeed, it follows directly from the definition that $\mathscr{F}_{\mathcal{R}}$ is closed under taking subquotients, and that $\mathscr{F}_{\mathcal{R}}$ has finitely many simple objects. For any simple objects $U, V\in \mathscr{F}_{\mathcal{R}}$, there exist $T, T^\prime\in \mathcal{S}_{\mathscr{E}_{S}}$ such that $U$ and $V$ are subquotients of $H^0F(T)$ and $H^0F(T^\prime)$, respectively. It follows that $U\otimes V$ is a subquotient of $H^0F(T) \otimes H^0F(T^\prime)$, and hence of $H^0F(T\otimes T^\prime)$. This means that $U\otimes V\in \mathscr{F}_{\mathcal{R}}.$ Similarly, by induction, one can show that $\mathscr{F}_{\mathcal{R}}$ is closed under taking tensor products. A similar argument shows that $\mathscr{F}_{\mathcal{R}}$ is closed under taking duality. Thus, $\mathscr{F}_{\mathcal{R}}$ is a tensor subcategory of $\mathscr{D}$ with finitely many simple objects.
By Lemma \ref{lem:finite}, we have $a(F^{-1}(M))=b(F^{-1}(M))=0$ for all nonzero $M\in\mathscr{F}_{\mathcal{R}}$, and
$$H^0F^{-1}\mid_{\mathscr{F}_{\mathcal{R}}}:\mathscr{F}_{\mathcal{R}}\rightarrow \mathscr{C}$$
is exact.
We next show that $X$ is a simple object in $\mathscr{D}$. Assume that $Y$ is a nonzero subobject of $X$. Then by Lemma \ref{lem:reduced}, $Y\otimes X^*$ is a nonzero subobject of $X\otimes X^*$. Since $H^0 F(\mathscr{E}_{S}) \subseteq \mathscr{F}_{\mathcal{R}}$, it follows that $Y\otimes X^*\in \mathscr{F}_{\mathcal{R}}$. Using the fact that $H^0F^{-1}\mid_{\mathscr{F}_{\mathcal{R}}}$ is exact, we obtain $H^0F^{-1}(Y\otimes X^*)$ is a subobject of $H^0F^{-1}(X\otimes X^*)$ ($\cong S\otimes S^*$). Let $j: Y\hookrightarrow X$ denote the standard inclusion, and consider
$$
f:=(\id_{S}\otimes \operatorname{ev}_{S})\circ(i\otimes \id_S): H^0F^{-1}(Y\otimes X^*)\otimes S\rightarrow S,
$$
where $$i:=H^0F^{-1}(j\otimes \id_{X^*}) :H^0F^{-1}(Y\otimes X^*)\hookrightarrow S\otimes S^*.$$
Clearly, $f\neq 0$, which implies that $f$ is an epimorphism.
It follows from \cite[Exercise 2.10.6]{EGNO15} that
\begin{equation*}
\begin{aligned}
f\otimes \id_{S^*}&=(\id_{S}\otimes \operatorname{ev}_{S}\otimes \id_{S^*})\circ(i\otimes \id_S\otimes \id_{S^*})\\
&=
H^0F^{-1}( \id_X \otimes \operatorname{ev}_X \otimes \id_{X^*}) \circ  H^0F^{-1}(j \otimes \id_{X^*} \otimes \id_{X}\otimes \id_{X^*} )\\
&=H^0F^{-1}((j \otimes \id_{X^*}) \circ (\id_Y \otimes \operatorname{ev}_X \otimes \id_{X^*}))\\
&=i \circ H^0F^{-1} (\id_Y \otimes \operatorname{ev}_X \otimes \id_{X^*}).
\end{aligned}
\end{equation*}
This implies that $i$ is an epimorphism, and hence an isomorphism. Consequently, $Y\cong X$. This completes the proof.
\end{proof}

It should be noted that in the above theorem, $\mathscr{E}_S$ is a tensor subcategory depending on $S$, and is therefore not fixed. Moreover, the above theorem includes \cite[Remark 4.6]{XZ26} (see also \cite[Corollary 4.7]{Yu26}) as a special case, which asserts that a monoidal derived equivalence between tensor categories with finitely many simple objects yields their tensor equivalence.

In \cite[Definition 4.14.5]{EGNO15}, the adjoint subcategory was defined for semisimple tensor categories. Motivated by this notion, we extend the terminology to any tensor category $\mathscr{C}$ (not necessarily semisimple) by defining its adjoint subcategory, denoted $\mathscr{C}_{ad}$, to be the smallest tensor Serre subcategory containing $S\otimes S^*$ for all $S\in\mathcal{S}_{\mathscr{C}}$. As a corollary of Theorem \ref{thm:SotimesS*}, we have:
\begin{corollary}
Let $\mathscr{C}$ and $\mathscr{D}$ be tensor categories, and let $F: \mathbf{D}^b(\mathscr{C})\rightarrow \mathbf{D}^b(\mathscr{D})$ be a monoidal derived equivalence. Suppose that $\mathscr{C}_{ad}$ has finitely many simple objects. Then $F$ is simple-shifted. Moreover, $\mathscr{C}$ is tensor equivalent to $\mathscr{D}.$
\end{corollary}

\begin{corollary}\label{coro:finiteotimesinvertible}
Let $\mathscr{C}$ and $\mathscr{D}$ be tensor categories, and let $F: \mathbf{D}^b(\mathscr{C})\rightarrow \mathbf{D}^b(\mathscr{D})$ be a monoidal derived equivalence. Suppose that every simple object $S\in\mathscr{C}$ admits a factorization $S \cong R\otimes T$ or $S\cong T\otimes R$, where $R$ is invertible and $T$ is a simple object belonging to a tensor subcategory $\mathscr{E}_{S}$ of $\mathscr{C}$ with finitely many isomorphism classes of simple objects. Then $F$ is simple-shifted. Moreover, $\mathscr{C}$ is tensor equivalent to $\mathscr{D}.$
\end{corollary}
\begin{proof}
Without loss of generality, we assume that every simple object $S\in\mathscr{C}$ satisfies $S\cong R\otimes T$, where $R$ is invertible and $T$ is a simple object belongs to a tensor subcategory $\mathscr{E}_{S}$ of $\mathscr{C}$ with finitely many isomorphism classes of simple objects.
By \cite[Exercise 2.10.7(b)]{EGNO15}, we have $$ S^*\otimes S\cong  T^*\otimes R^* \otimes R\otimes T \cong T^*\otimes T.$$  Since $T$ lies in a tensor subcategory with finitely many simple objects, $T^*\otimes T$ also lies in this subcategory.
Using Theorem \ref{thm:SotimesS*}, we can complete the proof.
\end{proof}

We provide the following example to illustrate that the conditions required in the above corollary are satisfied not only by pointed tensor categories but also by other tensor categories.
\begin{example}\rm
Let $\k$ be an algebraically closed field of characteristic $0$.
We construct a Hopf algebra of the form $(R_{2,2})^*\#\k\mathbb{Z}$., where $(R_{2,2})^*$ is the dual of the $8$-dimensional Radford Hopf algebra $R_{2,2}$ \cite{Rad75}. The structure of this dual can also be found in \cite{ACE15}. Specifically, let $H$ be a Hopf algebra over $\k$ generated as an algebra by $A, X, g, g^{-1}$ with relations
\begin{equation*}
A^4=gg^{-1}=g^{-1}g=1,\;\; X^2=0,\;\; AX=\sqrt{-1}XA,\;\; Ag=gA,\;\; Xg=gX.
\end{equation*}
Its coalgebra structure and antipode are determined respectively by
\begin{equation*}
    \begin{aligned}
 &\Delta(A)= A\otimes A+2XA^3\otimes XA,\;\; \varepsilon(A)=1,\;\;S(A)=A^3,\\
&\Delta(X)=X\otimes1+A^2\otimes X,\;\; \varepsilon(X)=0,\;\;S(X)=XA^2,\\
& \Delta(g)=g\otimes g,\;\;\Delta(g^{-1})=g^{-1}\otimes g^{-1},\;\;\varepsilon(g^{-1})=\varepsilon(g)=1,\;\;S(g)=g^{-1},\;\;S(g^{-1})=g.
\end{aligned}
\end{equation*}
Note that the Hopf subalgebra of $H$ generated by $A$ and $X$ is precisely the dual $(R_{2,2})^*$ of the $8$-dimensional Radford Hopf algebra.
The simple subcoalgebras of $H$ are given by $$\{\k g^i, \k A^2g^i \mid i\in\mathbb{Z} \}\cup \{C_i \mid i\in\mathbb{Z} \},$$
where the comatrix coalgebra corresponding to each $C_i$ is represented in matrix form as
$$
\begin{pmatrix}
Ag^i & \sqrt{2} XA^3g^i \\
\sqrt{2} XAg^i & A^3g^i
\end{pmatrix}.
$$
Clearly, $H$ is not pointed. Indeed, its coradical is not a Hopf subalgebra, which by definition means that $H$ fails to have the dual Chevalley property.
It is well known that the simple objects in the finite-dimensional comodule category $\mathscr{M}^{H}$ of $H$ are determined by its simple subcoalgebras (see \cite[Theorem 3.2.11]{Rad12}). We can directly prove that every simple object $S$ in $\mathscr{M}^{H}$ is isomorphic to $T \otimes R$, where $R$ is a one-dimensional simple comodule of $H$ and $T$ is a simple comodule of $(R_{2,2})^*$. According to Corollary \ref{coro:finiteotimesinvertible}, the monoidal triangulated equivalence between the monoidal derived categories of a tensor category $\mathscr{C}$ and $\mathscr{M}^H$ implies that $\mathscr{C}$ and $\mathscr{M}^H$ are tensor equivalent.
\end{example}

Recall that the \textit{Ext quiver} $\mathrm{Q}(\mathscr{A})$ of a $\k$-linear abelian category $\mathscr{A}$ is defined as an oriented graph with vertex set indexed by $\mathcal{S}_\mathscr{A}$, where for any $S, T \in \mathcal{S}_\mathscr{A}$, the number of arrows from $S$ to $T$ is given by $\dim_k \operatorname{Ext}^1_{\mathscr{A}}(S, T)$. By \cite[Exercise 1.5.10(i) and (ii)]{EGNO15}, there is a unique decomposition
$$
\mathscr{A}\cong \bigoplus_{\gamma\in \Gamma} \mathscr{A}^{\gamma},
$$
where each $\mathscr{A}^{\gamma}$ is a block of $\mathscr{A}$, and for every $\gamma \in \Gamma$, the Ext quiver $\mathrm{Q}(\mathscr{A}^{\gamma})$ of $\mathscr{A}^{\gamma}$ is connected.

The following theorem shows that the existence of simple-shifted monoidal derived equivalences does not actually require the entire tensor category to possess a global property; it suffices that each of its components satisfies a uniform property.
\begin{theorem}\label{thm:Cgamma}
Let $\mathscr{C}$ and $\mathscr{D}$ be tensor categories, and let $F: \mathbf{D}^b(\mathscr{C})\rightarrow \mathbf{D}^b(\mathscr{D})$ be a monoidal derived equivalence. Write $\mathscr{C}\cong \bigoplus_{\gamma\in \Gamma} \mathscr{C}^{\gamma}.$ Suppose that, for each $\gamma \in \Gamma$, at least one of the following conditions holds.
\begin{itemize}
    \item [(i)] For every simple object $S\in \mathscr{C}^\gamma$, there exists a tensor subcategory $\mathscr{E}_{S}$ of $\mathscr{C}$ with finitely many simple objects such that either $S\otimes S^* \in \mathscr{E}_{S}$ or $S^*\otimes S\in \mathscr{E}_{S}$.
        \item[(ii)] For every simple object $S\in \mathscr{C}^\gamma$, we have $S\cong S^*.$
    \item [(iii)] There exists a simple object $S^\gamma\in \mathscr{C}^{\gamma}$ such that every simple object $S\in \mathscr{C}^{\gamma}$ is of the form $S\cong L_S\otimes S^\gamma \otimes R_S$, where $L_S, R_S\in \mathscr{C}$ are invertible objects satisfying $a(F(L_S\otimes R_S))=b(F(L_S\otimes R_S))=0$.
\end{itemize}
Then $F$ is simple-shifted. Moreover, $\mathscr{C}$ is tensor equivalent to $\mathscr{D}.$
\end{theorem}
\begin{proof}
Our idea is to prove that $F^{-1}$ is simple-shifted. Once this is established, Lemma \ref{lem:psibijective} allows us to prove that $F$ is also simple-shifted.
For any simple object $U\in\mathscr{D}$, let $K^\bullet=F^{-1}(U)$. Since $F$ is fully faithful, it follows that
$$
\operatorname{End}_{\mathbf{D}^b(\mathscr{C})}(K^\bullet)\cong \operatorname{End}_{\mathbf{D}^b(\mathscr{D})}(U)\cong \operatorname{End}_{\mathscr{D}}(U).
$$
By \cite[Lemma 1.5.2]{EGNO15}, $\operatorname{End}_{\mathbf{D}^b(\mathscr{C})}(K^\bullet)$ is a division algebra. Consequently, $K^\bullet$ is indecomposable. 
Note that each component $K^j$ of the complex $K^\bullet$ admits a block decomposition
$$
K^j\cong \bigoplus_{\gamma\in\Gamma} K^{j, \gamma},
$$
where $K^{j, \gamma} \in \mathscr{C}^{\gamma}.$
The indecomposability of $K^\bullet$ forces it to be isomorphic to an object in $\mathbf{D}^b(\mathscr{C})$ with the property that all of its components lie in the same block $\mathscr{C}^{\gamma_0}$ for some $\gamma_0\in\Gamma$. Without loss of generality, we may assume that $K^\bullet$ itself is such an object. The validity of Theorem \ref{thm:SotimesS*} ensures that we only need to consider the case where $\mathscr{C}^{\gamma_0}$ satisfies conditions (ii) or (iii).

Assume that condition (ii) holds. Let $T$ be a simple quotient object of $H^{b(K^\bullet)}(K^\bullet)$. Let $f$ denote the composition of the following maps:
\begin{equation*}
H^{b(K^\bullet)}(K^\bullet) \twoheadrightarrow T \cong T^* \hookrightarrow (H^{b(K^\bullet)}(K^\bullet))^*\cong H^{-b(K^\bullet)}((K^\bullet)^*).
\end{equation*}
Clearly, we have $f\neq 0.$
Moreover, we have the following distinguished triangles
$$
H^{a(K^\bullet)}(K^\bullet)[-a(K^\bullet)]\rightarrow K^\bullet\rightarrow \tau_{\geq a(K^\bullet)+1}K^\bullet\rightarrow H^{a(K^\bullet)}(K^\bullet)[1-a(K^\bullet)]
$$
and
$$
\tau_{\leq b(K^\bullet)-1}(K^\bullet)\rightarrow  K^\bullet\rightarrow H^{b(K^\bullet)}(K^\bullet)[-b(K^\bullet)] \rightarrow (\tau_{\leq b(K^\bullet)-1}(K^\bullet))[1].
$$
Let $g$ denote the composition of the following maps:
\begin{equation*}
K^\bullet\rightarrow H^{b(K^\bullet)}(K^\bullet)[-b(K^\bullet)] \xrightarrow{f[-b(K^\bullet)]} H^{-b(K^\bullet)}((K^\bullet)^*)[-b(K^\bullet)]\rightarrow (K^\bullet)^*[-2b(K^\bullet)].
\end{equation*}
We have $H^{b(K^\bullet)}(g)=f,$ which forces $g\neq0$. Since $F(g)\in \operatorname{Hom}_{\mathbf{D}^b(\mathscr{D})}(U, U^*[-2b(K^\bullet)])$, it follows that $b(K^\bullet)\leq0$. A similar argument shows that $a(K^\bullet)\geq 0$. Thus, we have $$a(K^\bullet)=b(K^\bullet)=0.$$ We claim that $H^0(K^\bullet)$ is simple. Indeed, the isomorphism $\operatorname{Hom}_{\mathbf{D}^b(\mathscr{D})}(U, U^*) \cong \operatorname{Hom}_{\mathscr{D}}(U, U^*)$ guarantees that $F(f)$ is an isomorphism. It follows that $f$ is an isomorphism. Hence, $H^0(K^\bullet)$ is a simple object.

Provided that condition (iii) is satisfied, we show that $a(K^\bullet)=b(K^\bullet)$. Assume for contradiction that $a(K^\bullet)< b(K^\bullet)$. The cohomology objects $H^{a(K^\bullet)}(K^\bullet)$ and $H^{b(K^\bullet)}(K^\bullet)$ admit a simple subobject $S\in\mathscr{C}$ and a simple quotient object $T\in\mathscr{C}$, respectively. Clearly, $S, T\in \mathscr{C}^{\gamma_0}.$ Hence, one can find invertible objects $L, R\in\mathscr{C}$ for which $$T\cong L\otimes S\otimes R$$ and $$a(F(L\otimes R))=b(F(L\otimes R))=0.$$
Let $f^\prime$ denote the composition of the following maps:
\begin{equation*}
  H^{b(K^\bullet)} (K^\bullet) \twoheadrightarrow T\cong  L\otimes S\otimes R \hookrightarrow L\otimes H^{a(K^\bullet)}(K^\bullet) \otimes R.
\end{equation*}
It follows that $f^\prime\neq 0.$
Let $g^\prime$ denote the composition of the following maps:
\begin{equation*}
\begin{aligned}
&K^\bullet \rightarrow H^{b(K^\bullet)}(K^\bullet)[-b(K^\bullet)] \xrightarrow{f^\prime[-b(K^\bullet)]}   (L\otimes H^{a(K^\bullet)}(K^\bullet) \otimes R) [-b(K^\bullet)] \\
&\rightarrow (L\otimes K^{\bullet}\otimes R)[a(K^\bullet)-b(K^\bullet)].
\end{aligned}
\end{equation*}
Since $H^{b(K^\bullet)}(g^\prime)=f^\prime\neq 0$, it follows that $g^\prime\neq 0.$
Note that we have
\begin{equation*}
    \begin{aligned}
         &\operatorname{Hom}_{\mathbf{D}^b(\mathscr{C})}(K^\bullet, L\otimes K^\bullet \otimes R[a(K^\bullet)-b(K^\bullet)])\\
         \cong&
         \operatorname{Hom}_{\mathbf{D}^b(\mathscr{D})}(U, (F(L)\otimes U \otimes F(R))[a(K^\bullet)-b(K^\bullet)])\\
         =&0,
    \end{aligned}
\end{equation*}
where the last equality follows from the fact that $(\mathbf{D}^{\leq 0}_{\mathscr{D}},\mathbf{D}^{\geq 0}_{\mathscr{D}})$ is a standard $t$-structure. This leads to a contradiction. Thus, we have $$F^{-1}(U)=H^{a(K^\bullet)}(K^\bullet)[-a(K^\bullet)].$$For simplicity, we denote $H^{a(K^\bullet)}(K^\bullet)$ by $X$. It remains to show that $X$ is simple. Suppose $V$ is a simple subobject and $W$ is a simple quotient object of $X$, with $V, W\in\mathcal{S}_{\mathscr{C}}$. There exist invertible objects $L^\prime, R^\prime \in\mathscr{C}$ such that $$W\cong L^\prime \otimes V\otimes R^\prime$$
and $$a(F(L^\prime\otimes R^\prime))=b(F(L^\prime\otimes R^\prime))=0.$$
Let $h^\prime$ be defined as the composition of the maps below:
$$
X\twoheadrightarrow W \cong L^\prime \otimes V\otimes R^\prime\hookrightarrow L^\prime \otimes X\otimes R^\prime.
$$
It follows that $h^\prime\neq 0.$ Then we have
\begin{equation*}
 F(h^\prime): U [a(K^\bullet)] \rightarrow (F(L^\prime)  \otimes U\otimes  F(R^\prime))[a(K^\bullet)].
\end{equation*}
According to Lemma \ref{lem:invertible}, $$F(L^\prime)  \otimes U\otimes  F(R^\prime)\cong L^{\prime\prime} \otimes U \otimes  R^{\prime\prime}$$ for some invertible objects $L^{\prime\prime}, R^{\prime\prime}\in\mathscr{D}$.
The object $L^{\prime\prime} \otimes U \otimes R^{\prime\prime}$ is simple, and thus $$U\cong L^{\prime\prime} \otimes U \otimes R^{\prime\prime}.$$
Since \begin{equation*}
\begin{aligned}
\operatorname{Hom}_{\mathbf{D}^b(\mathscr{C})}(X, L^{\prime} \otimes X \otimes  R^{\prime})
\cong
\operatorname{Hom}_{\mathbf{D}^b(\mathscr{D})}(U, L^{\prime\prime} \otimes U \otimes  R^{\prime\prime})
\cong \operatorname{Hom}_{\mathscr{D}}(U, L^{\prime\prime} \otimes U \otimes  R^{\prime\prime}),
\end{aligned}
\end{equation*}
it follows that $h^\prime$ is an isomorphism. Note that $\operatorname{Im}(h^\prime)\cong W$, which implies that $X\cong W$. Thus, $F^{-1}$ maps a simple object to the shift of a simple object.
\end{proof}

Observe that when $\mathscr{C}$ is semisimple, every block consists of a single simple object $S$. The isomorphism $S\cong \mathbf{1}\otimes S\otimes \mathbf{1}$ holds trivially, which implies that the above theorem generalizes \cite[Theorem 4.4]{Tan26}. We can also give non-semisimple examples satisfying conditions (ii) and (iii) of the above theorem.
\begin{example}
Let $\k$ be an algebraically closed field of characteristic $0$. We construct a Hopf algebra of the form $\k\langle x\rangle\# H(e_{\pm 1}, f_{\pm 1}, u, v)_0$, where $H(e_{\pm 1}, f_{\pm 1}, u, v)_0$ is the coradical of the Hopf algebra $H(e_{\pm 1}, f_{\pm 1}, u, v)$ introduced in \cite[Definition 5.1]{YL26}. More explicitly, let $H$ be a Hopf algebra over $\k$ generated as an algebra by $x, e_i, f_i$ for $i\in\mathbb{Z}$, subject to the following relations
\begin{equation*}
      1=e_0+f_0, \;\;e_ie_j=e_{i+j},\;\;f_if_j=f_{i+j},\;\;e_if_j=f_je_i=0,\;\;  e_i x=xe_i,\;\; f_i x= xf_i.
\end{equation*}
The comultiplication, counit and the antipode are given by
\begin{equation*}
    \begin{aligned}
        &\Delta(e_i)=e_i\otimes e_i+f_i\otimes f_{-i},\;\;\varepsilon(e_i)=1,\;\;S(e_i)=e_{-i},\\
        &\Delta(f_i)=e_i\otimes f_i+f_i\otimes e_{-i},\;\;\varepsilon(f_i)=0,\;\;S(f_i)=f_{i},\\
        &\Delta(x)=1\otimes x+x \otimes 1,\;\; \varepsilon(x)=0,\;\; S(x)=-x.
    \end{aligned}
\end{equation*}
For each $i\geq 1$, let $C_i=\operatorname{span}\{e_i, f_i, e_{-i}, f_{-i}\}$. We can show that each $C_i$ is a simple subcoalgebra.
Note that $H$ is a Hopf algebra with the dual Chevalley property, whose coradical is given by
$$
H_0=\k 1\oplus \k (e_0-f_0)\oplus \bigoplus_{i\geq 1}C_i.
$$
The Grothendieck ring structure of the category of $H_0$-comodules ensures that $H$ does not satisfy the conditions in Theorem \ref{thm:SotimesS*}. See \cite[Section 5]{YL26} for details on the Grothendieck ring structure of $\mathscr{M}^{H_0}$.
In our previous work \cite[Sections 3 and 4]{YLL24}, we studied the Ext quiver of the finite-dimensional comodule category of Hopf algebras with the dual Chevalley property. We can show that the Ext quiver of $\mathscr{M}^H$ is of the following form:
$$
\begin{tikzpicture}
\filldraw [black] (4,-0.7) circle (2pt)node[anchor=west]{};
\filldraw [black] (6,-0.7) circle (2pt)node[anchor=west]{};
\filldraw [black] (8,-0.7) circle (2pt)node[anchor=west]{};
\filldraw [black] (10,-0.7) circle (2pt)node[anchor=west]{};
\filldraw [black] (11,-0.7) circle (0pt)node[anchor=west]{$\cdots$};
\draw[thick, ->] (3.9,-0.6) arc (0:340:0.6);
\draw[thick, ->] (5.9,-0.6) arc (0:340:0.6);
\draw[thick, ->] (7.9,-0.6) arc (0:340:0.6);
\draw[thick, ->] (9.9,-0.6) arc (0:340:0.6);
\end{tikzpicture}
$$
A direct verification shows that $H$ satisfies condition (iii). Whenever there is a monoidal triangulated equivalence between the monoidal derived categories of a tensor category $\mathscr{C}$ and $\mathscr{M}^H$, Theorem \ref{thm:Cgamma} ensures that $\mathscr{C}$ and $\mathscr{M}^H$ are tensor equivalent. Moreover, we know that $S(C_i)=C_i$ for all $i\geq 1.$ For all Hopf algebras $H^\prime$ with coradical $H_0$, one can directly prove that every simple object in their finite-dimensional comodule category satisfies condition (ii). A specific example is given by taking $H^\prime$ to be $H(e_{\pm 1}, f_{\pm 1}, u, v)$ as introduced in \cite[Definition 5.1]{YL26}.
\end{example}

We point out that, for a coquasi-Hopf algebra with the dual Chevalley property, the Ext quiver of its comodule category and the coalgebra structures of its connected components have been characterized in our previous work \cite{Yu26b}. Furthermore, in that work, we introduced the notion of a generalized Hopf quiver and used quiver-theoretic methods to construct (coquasi-)Hopf algebras with the dual Chevalley property. Consequently, examples of tensor categories satisfying the conditions of Theorem \ref{thm:Cgamma} can be found in abundance among (coquasi-)Hopf algebras with the dual Chevalley property. 

\section{Applications to quantum groups}\label{section5}
In this section, for convenience, we take $\k$ to be the complex field $\mathbb{C}$. Let $\mathfrak{g}$ be a finite-dimensional complex semisimple Lie algebra of rank $n$. Let $q \in \mathbb{C}$ be a nonzero element. If $q$ is a root of unity, we assume that its order $l>1$ is odd and that $l$ is relatively prime to $3$ whenever $\mathfrak{g}$ has a simple component of type $G_2.$ We apply the results of Section \ref{section4} to the categories of finite-dimensional representations of the quantized enveloping algebra $U_q(\mathfrak{g})$, the quantized Borel subalgebra $U_q^{\geq 0}(\mathfrak{g})$, and the quantized coordinate ring $\mathcal{O}_{q}(G)$. We prove that any tensor category monoidally derived equivalent to the finite-dimensional representation category of one of these Hopf algebras is tensor equivalent to it.

Suppose that $q$ is not a root of unity. The quantized enveloping algebra $U_q(\mathfrak{g})$ was introduced by Drinfel'd \cite{Dri85} and Jimbo \cite{Jim85}. This is an infinite-dimensional Hopf algebra with generators elements $E_i, F_i, K_i^{\pm 1}$, $i=1,2,\dots,n$. According to \cite[Theorem 5.12]{JL92}, the category of finite-dimensional representations of $U_q(\mathfrak{g})$ is semisimple. Consequently, by \cite[Theorem 4.4]{Tan26} (also, Theorem \ref{thm:Cgamma}), every tensor category monoidally derived equivalent to this representation category is tensor equivalent to it.

The quantized Borel subalgebra $U_q^{\geq 0}(\mathfrak{g})$ is the subalgebra of $U_q(\mathfrak{g})$ generated by $E_i, K_i^{\pm 1},$ $i=1,2,\dots,n$. By \cite[Theorem 11.5]{Jos94}, all finite-dimensional simple modules of $U_q^{\geq 0}(\mathfrak{g})$ are \(1\)-dimensional. Therefore, the finite-dimensional representation category of $U_q^{\geq 0}(\mathfrak{g})$ is pointed. As a corollary of Proposition \ref{prop:pointed}, if a tensor category is monoidally derived equivalent to this representation category, it is tensor equivalent to it. Let $G$ be a simply connected complex semisimple algebraic group with Lie algebra $\mathfrak{g}$. By \cite[Proposition 4.4]{Jos94}, the finite-dimensional representation category of the quantized coordinate ring $\mathcal{O}_{q}(G)$ (see, for example, \cite[Chapter 7]{Jan96}) is also pointed. Therefore, an analogous assertion holds for $\mathcal{O}_{q}(G)$.

Next, we consider the case where $q$ is a primitive root of unity. Inspired by \cite{HQW26,HQWZ26}, we first consider the case where $H$ is a Hopf algebra over $\k$ with a central Hopf subalgebra $C$ such that $H$ is a finitely generated $C$-module. Put $A:=H/\ker(\varepsilon|_{C})H$, and let ${}_{A}\mathscr{M}$ denote the category of finite-dimensional representations of $A$. By \cite[Theorem 2.3(2)]{BC21}, $I:=\ker(\varepsilon|_{C})H$ is a Hopf ideal, and $A$ is a finite-dimensional Hopf algebra. Through the quotient map, any finite-dimensional $A$-module can be naturally viewed as an $H$-module. Equivalently, ${}_{A}\mathscr{M}$ can be identified with the full subcategory
$$
\{M\in {}_{H}\mathscr{M}\mid I \cdot M=0\}.
$$
We have the following lemma.
\begin{lemma}\label{lem:(H,C)}
Let $H$ be a Hopf algebra over $\mathbb{C}$, and let $C$ be a central Hopf subalgebra of $H$ such that $H$ is a finitely generated $C$-module. Then for any finite-dimensional simple $H$-module $V$, the tensor products $V\otimes V^*$ and $V^*\otimes V$ both lie in ${}_{A}\mathscr{M}$.
\end{lemma}
\begin{proof}
Let $V$ be a finite-dimensional simple $H$-module. For any $c\in C$, define the linear map $T_c: V\to V$ by $T_c(v)=c\cdot v$. Because $C$ is central, $T_c$ is an $H$-module endomorphism. By Schur's lemma, there exists a scalar $\chi(c)\in\mathbb{C}$ such that $T_c(v)=\chi(c)v$ for all $v\in V$. It is clear that $\chi$ defines an algebra homomorphism from $C$ to $\mathbb{C}$. For any $v\otimes f\in V\otimes V^*$ and $c\in C$,
\begin{equation*}
c\cdot (v\otimes f)=\chi(c_1) v \otimes \chi(S(c_2)) f=\chi(c_1S(c_2))v\otimes f=\varepsilon(c)v\otimes f,
\end{equation*}
where Sweedler notation $\Delta(c)=c_1 \otimes c_2$ is applied to the coproduct. Consequently, $$I\cdot (V\otimes V^*)=0,$$ which implies $V\otimes V^*\in {}_{A}\mathscr{M}$. A similar argument shows that $V^*\otimes V \in {}_{A}\mathscr{M}$.
\end{proof}

As a corollary of Theorem \ref{thm:Cgamma} and Lemma \ref{lem:(H,C)}, we have
\begin{corollary}\label{coro:(H,C)}
Let $H$ be a Hopf algebra over $\mathbb{C}$ with a bijective antipode, and let $C$ be a central Hopf subalgebra of $H$ such that $H$ is a finitely generated $C$-module. Then every tensor category monoidally derived equivalent to ${}_{H}\mathscr{M}$ is tensor equivalent to ${}_{H}\mathscr{M}$.
\end{corollary}

With the above corollary, we are now in a position to treat the De Concini-Kac quantized enveloping algebra $U_q(\mathfrak{g})$ \cite{DK89}, the De Concini-Kac-Procesi quantized Borel subalgebra $U_q^{\geq 0}(\mathfrak{g})$ \cite{DP93}, and the quantized coordinate ring $\mathcal{O}_q(G)$ \cite{DL94} when $q$ is a root of unity. Note that their representation categories no longer possess the Chevalley property; that is, the tensor product of simple objects is not necessarily semisimple. See \cite[Theorem 0.1 and Remark 0.5]{HQW26}.

According to \cite[Corollary and Theorem 19.1]{DP93} (see also \cite[Subsection 3.2]{BC21}), $U_q(\mathfrak{g})$ satisfies the conditions of Corollary \ref{coro:(H,C)}. Likewise, $U_q^{\geq 0}(\mathfrak{g})$ satisfies the conditions of Corollary \ref{coro:(H,C)} by \cite[Proposition 1.7(c)]{DK89} and \cite[Proposition 5.6(d)]{DKP92} (see also \cite[Subsection 2.4]{HQWZ26}). For $\mathcal{O}_q(G)$, the fact that it satisfies the conditions of Corollary \ref{coro:(H,C)} follows from \cite[Proposition 6.4 and Theorem 7.2]{DL94} (see also \cite[Example 2.4]{HQW26}).

The above analysis yields the following conclusion.
\begin{proposition}\label{prop:quantum}
Let $q \in \mathbb{C}$ be a nonzero element. If $q$ is a root of unity, we assume that its order $l>1$ is odd and that $l$ is relatively prime to $3$ whenever $\mathfrak{g}$ has a simple component of type $G_2.$
Then any tensor category monoidally derived equivalent to the finite-dimensional representation category of the quantized enveloping algebra $U_q(\mathfrak{g})$, the quantized Borel subalgebra $U_q^{\geq 0}(\mathfrak{g})$, or the quantized coordinate ring $\mathcal{O}_{q}(G)$ is tensor equivalent to the corresponding category.
\end{proposition}

\section*{Acknowledgement}
AI was used solely for language polishing and spell-checking. The author takes full responsibility for the content of the paper. The author would like to thank Yimin Huang, Tiancheng Qi, Quanshui Wu, and Zhongkai Mi for their inspiring talks, which provided insightful ideas for Section 5. The author is also grateful to Tiancheng Qi for his helpful comments on Section 5.
Yu was supported by the China Postdoctoral Science Foundation [Grant number 2026M793410] and the Postdoctoral Fellowship Program of China Postdoctoral Science Foundation [Grant number GZC20261673].

\end{document}